\documentclass[english,leqno,10pt,a4paper]{article}
\usepackage{amsmath,amstext, amsthm, amssymb}

\usepackage{pifont}
\usepackage[hypertex]{hyperref}
\usepackage{srcltx}
\usepackage[totalheight=24 true cm, totalwidth=17 true cm]{geometry}
\usepackage[english]{babel}
\usepackage[all]{xy}

\usepackage{t1enc}
\usepackage{color}

\newtheorem{thm}{Theorem}[section]
\newtheorem{cor}[thm]{Corollary}
\newtheorem{lemma}[thm]{Lemma}
\newtheorem{prop}[thm]{Proposition}

\theoremstyle{remark}
\theoremstyle{definition}
\newtheorem{defn}[thm]{Definition}

\newtheorem{rmk}[thm]{Remark}
\newtheorem{exa}[thm]{Example}

\numberwithin{equation}{section}

\def\beq{\begin{equation}}
\def\eeq{\end{equation}}

\def\crash#1{}

\def\N{{\mathbb N}}
\def\Z{{\mathbb Z}}

\def\R{{\mathbb R}}
\def\C{{\mathbb C}}

\def\l{\left}
\def\r{\right}
\def\[[{\l[\l[}
\def\]]{\r]\r]}
\def\p{\prime}

\def\sgq{\sigma_q}

\def\ie{\emph{i.e.~}}
\def\ds{\displaystyle}

\def\cM{{\mathcal M}}

\def\cR{{\mathcal R}}
\def\cS{{\mathcal S}}

\def\wtilde{\widetilde}

\def\ul{\underline}
\def\ol{\overline}

\def\a{\alpha}
\def\be{\beta}
\def\de{\delta}

\def\sg{\sigma}

\def\la{\lambda}

\def\Sg{\Sigma}
\def\De{\Delta}

\def\kI{\mathfrak{I}}
\def\kJ{\mathfrak{J}}

\def\deg{\mathop{\rm deg}}

\author{Lucia Di Vizio and Charlotte Hardouin}
\date{\today}

\title{Descent for differential Galois theory of difference equations\\
Confluence and $q$-dependency
\footnotetext{Date: \today}
\footnotetext{Lucia DI VIZIO,
{Institut de Math\'{e}matiques de Jussieu,
Topologie et g\'{e}om\'{e}trie alg\'{e}briques,}
{Case 7012, 2, place Jussieu, 75251 Paris Cedex 05, France.}
{e-mail: {\tt divizio@math.jussieu.fr}.}}
\footnotetext{Charlotte HARDOUIN, Institut de Math\'{e}matiques de Toulouse,
118 route de Narbonne,
31062 Toulouse Cedex 9, France.
{e-mail: {\tt hardouin@math.ups-tlse.fr}}.}
\footnotetext{Work supported by the project ANR- 2010-JCJC-0105 01 qDIFF}}

\begin{document}
\maketitle
\bibliographystyle{alpha}

\begin{abstract}
The present paper essentially contains two results that generalize and improve
some of the constructions of \cite{HardouinSinger}. First of all, in the case of one derivation,
we prove that the parameterized Galois theory for difference equations constructed
in \cite{HardouinSinger}
can be descended from a differentially closed to an algebraically closed field.
In the second part of the paper, we show that the theory can be applied to deformations of $q$-series
to study the differential dependency
with respect to $x\frac{d}{dx}$ and $q\frac{d}{dq}$.
We show that
the parameterized difference Galois group (with respect to a convenient derivation defined in the text)
of the Jacobi Theta function
can be considered as the Galoisian counterpart of the heat equation.
\end{abstract}

\setcounter{tocdepth}{2}
\tableofcontents
\subsection*{Introduction}
\addcontentsline{toc}{section}{Introduction}

The present paper essentially contains two results that generalize and improve some of
the constructions  of \cite{HardouinSinger}. First of all, in the case of one derivation, we perform a descent of the constructions in \cite{HardouinSinger}
from a differentially closed to an algebraically closed field. The latter being much smaller, this is a good help
in the applications. In the second part of the paper, we show that the theory can be applied to deformations of $q$-series,
which appear in many settings like quantum invariants, modular forms, ..., to study the differential dependency
with respect to $x\frac{d}{dx}$ and $q\frac{d}{dq}$.
\par
In \cite{HardouinSinger}, the authors construct a specific Galois theory to study the differential relations among solutions
of a difference linear system.
To do so, they attach to a linear difference system a differential Picard-Vessiot ring, \ie a differential splitting ring, and,
therefore, a differential algebraic group, in the sense of Kolchin, which we will call the Galois $\De$-group. Roughly, this is a matrix group
defined as the zero set of algebraic differential equations.
In \cite{HardouinSinger}, both the differential Picard-Vessiot ring
and the Galois $\De$-group are proved to be well defined
under the assumption that the difference operator and the derivations commute with each other and that
the field of constants for the difference operator is differentially closed. The differential closure of a differential
field $K$ is an enormous field that contains a solution of any system of algebraic differential equations with coefficients in $K$ that has a solution in
some differential extension of $K$. When one works with $q$-difference equations,
the subfield of the field of meromorphic functions over $\C^*$ of
constants for the homothety $x\mapsto qx$ is the field of elliptic functions: its differential closure
is a very big field. The same happens for the shift $x\mapsto x+1$, whose field of constants are periodic functions.
In the applications, C. Hardouin and M. Singer prove that one can always descend the Galois $\De$-group with an \emph{ad hoc} argument.
Here we prove that, in the case of one derivation, we can actually suppose that the field of constants is algebraically closed and that the Galois $\De$-group descends
from a differentially closed field to an algebraically closed one (see Proposition \ref{prop:descgr}).
We also obtain that the properties and the results used in the applications descend to an algebraically closed field,
namely (see \S\ref{subsubsec:descent}):
\begin{itemize}
\item
the differential transcendence degree of an extension generated by a
fundamental solution matrix of the difference equation is equal to the differential dimension of the
Galois $\De$-group (see Proposition \ref{prop:degtrs});
\item
the sufficient and necessary condition for solutions of rank 1 difference equations
to be differentially transcendental (see Proposition \ref{prop:gatrans});
\item
the sufficient and necessary condition for a difference system to admit a linear differential
system totally integrable with the difference system
(see Corollary \ref{cor:intergrabilitydescent}).
\end{itemize}
The proof of the descent (see Proposition \ref{prop:pvdesc}) is based on an idea of M. Wibmer, which he used in a parallel difference setting
(see \cite{Wibmchev}, \cite{Cohn}).
The differential counterpart of his method is the differential prolongation of ideals, which goes back at least to E. Cartan, E. Vessiot, E.R. Kolchin
and more recently to B. Malgrange in the nonlinear theory, but
has never been exploited in this framework.
Differential prolongations have been used
in \cite{GranierFourier} to adapt the ideas of B. Malgrange to build a Galois theory for nonlinear $q$-difference equations, and
in \cite{diviziohardouin} to prove a comparison theorem between Granier's groupoid and the parameterized Galois group of \cite{HardouinSinger}.
So it appears that the differential prolongations are a common denominator of several differential and difference Galois theories.
\par
In the second part of the paper, we show that the theory applies to investigate the differential relations
with respect to the derivations $\frac{d}{dx}$ and $\frac{d}{dq}$
among solutions of $q$-difference equations over the field $k(q,x)$ of rational functions
with coefficients in a field $k$ of zero characteristic.
For instance, one can consider the Jacobi theta function
$$
\theta_q(x)=\sum_{n\in\Z} q^{n(n-1)/2}x^n,
$$
which is a solution of the $q$-difference equation $y(qx)=qxy(x)$. Let
$$
\ell_q=x\frac{\theta_q^\p(x)}{\theta_q(x)}
\hbox{~and~}\de=\ell_q(x)x\frac{d}{dx}+q\frac{d}{dq}.
$$
In this specific case, we show below that
the Galois $\de$-group can be considered as the Galoisian counterpart of the heat equation.
We show that, for solutions of rank $1$ equations $y(qx)=a(x)y(x)$ with $a(x)\in k(q,x)$, the following facts are equivalent (see Proposition \ref{thm:qconfluencerank1}):
\begin{itemize}
\item
$a(x) = \mu  x^r \frac{g(qx)}{g(x)}$, for some $r \in \Z$, $g \in k(q,x)$ and $\mu \in k(q)$;
\item
satisfying an algebraic $\frac{d}{dx}$-relation (over some field we are going to specify in the text below);
\item
satisfying
an algebraic $\de$-relation.
\end{itemize}
In the higher rank case, we  deduce from a general statement
a necessary and sufficient condition
on the Galois group so that a system $Y(qx)=A(x)Y(x)$ can be completed
in  a compatible (\ie integrable) way, by two
linear differential systems in $\frac{d}{dx}$ and $x\ell_q\frac{d}{dx}+q\frac{d}{dq}$
(see Proposition \ref{prop:intergrabilitydescent} and Corollary \ref{cor:qconf}).
The condition consists in the property of the group of being, up to a conjugation,
contained in the subgroup of constants of $GL$,
that is the differential subgroup of $GL$ whose points have coordinates that are annihilated
by the derivations. A result of P. Cassidy \cite{cassdiffgr} says that, for a
proper Zariski dense differential subgroup of a simple linear algebraic group, this is always the case.

{\bf Acknowledgements.}
We would like to thank D. Bertrand for the interest he has always shown in our work and his
constant encouragement, E. Hubert stimulating discussions and A. Ovchinnikov, M. Singer, M. Wibmer and the anonymous referee
for their attentive reading of the manuscript and their comments.
We are particularly indebted to M. Wibmer for the proof Proposition \ref{prop:pvdesc}.

\subsection*{Breviary of difference-differential algebra}

For the reader convenience, we briefly recall here some basic definitions of differential and difference algebra,
that we will use all along the text.
See \cite{Cohn:difference} and \cite{Levin:difference} for the general theory.

\medskip
By \emph{$(\sg,\De)$-field} we will mean a field $F$ of zero characteristic,
equipped with an automorphism $\sg$ and a set of commuting derivations
$\De=\{\partial_1,\dots,\partial_n\}$, such that $\sg\partial_i=\partial_i\sg$
for any $i=1,\dots,n$.
We will use the terms \emph{$(\sg,\De)$-ring, $\De$-ring, $\De$-field,
$\partial$-ring, $\partial$-field, for $\partial\in\De$,
$\sg$-ring, $\sg$-field, ...}, with the evident meaning, analogous to the definition of $(\sg,\De)$-field.
Notice that, if $F$ is a $(\sg,\De)$-field, the subfield $K=F^\sg$ of $F$ of the invariant elements with respect to $\sg$
is a $\De$-field. We will say that a $\De$-field, and in particular $K$, is \emph{$\De$-closed} if
any system of algebraic differential equations
in $\partial_1,\dots,\partial_n$ with coefficients in $K$ having a solution in an extension of $K$
has a solution in $K$ (see \cite{mcgrail}).
\par
A \emph{$(\sg,\De)$-extension of $F$} is a ring extension $\cR$ of $F$ equipped
with an extension of $\sg$ and of the derivations of $\De$, such that the commutativity conditions are preserved.
\par
A \emph{$(\sg,\De)$-ideal} of a $(\sg,\De)$-ring is an ideal that is invariant by both $\sg$ and
the derivations in $\De$. A \emph{maximal} $(\sg,\De)$-ideal is an ideal which is maximal with respect to the
property of being a $(\sg,\De)$-ideal. Similar definitions can be given for $\De$-ideals.
\par
A ring of \emph{$\De$-polynomials} with coefficients in $F$ is a ring of polynomials in infinitely many variables
$$
F\{X_1,\dots,X_\nu\}_{\De}:=F\l[X_{i}^{(\ul\a)}; i=1,\dots,\nu\r],
$$
equipped with the differential structure such that
$$
\partial_k(X_{i}^{(\ul\a)})=X_{i}^{(\ul\a+e_k)},
$$
for any $i=1,\dots,\nu$, $k=1,\dots,n$, $\ul\a=(\a_1,\dots,\a_n)\in\Z_{\geq 0}^n$,
with $\ul\a+e_k=(\a_1,\dots,\a_k+1,\dots,\a_n)$.
\par
We will need the ring of \emph{rational $\Delta$-functions over $GL_\nu(F)$} . This
is a localization of a ring of $\De$-polynomials in the variables $X_{i,j}$, for $ij,=1,\dots,\nu$:
$$
F\l\{X_{i,j},\,i,j=1,\dots,\nu;\frac{1}{\det(X_{i,j})}\r\}_{\De}:=
F\l\{X_{i,j};\,i,j=1,\dots,\nu\r\}_\De\l[\frac{1}{\det(X_{i,j})}\r],
$$
equipped with the induced differential structure.
We write $X$ for $(X_{i,j})$,
$\partial^{\ul\a}(X)$ for $\partial_1^{\a_1}\cdots\partial_n^{\a_n}(X)$
and $\det X$ for $\det(X_{i,j})$.
If $\Delta$ is empty, the ring $F\{X,\det X^{-1}\}_{\De}$ is nothing else than the ring of rational functions
$F[X,\det X^{-1}]$.
By \emph{$\De$-relation (over $F$)} satisfied by a given matrix $U$ with coefficients in a $(\sg,\De)$-extension $R$ of $F$,
we mean
an element of $F\{X,\det X^{-1}\}_\De$
that vanishes at $U$. The $\De$-relations satisfied by a chosen $U$ form a $\De$-ideal.

\section{Differential Galois theory for difference equations}

\subsection{Introduction to differential Galois theory for difference equations}
\label{sec:rapdifgal}

In this section, we briefly recall some results of \cite{HardouinSinger}.
Let $F$ be a field of zero characteristic equipped with an automorphism
$\sg$.
We denote the subfield of $F$ of all
$\sg$-invariant elements by $K = F^{\sg}$.

\begin{defn}
A \emph{($\sg$-)difference module $\cM=(M,\Sg)$ over $F$}
(also called \emph{$\sg$-module over $F$} or a \emph{$F$-$\sg$-module}, for short)
is a finite-dimensional $F$-vector space $M$ together with a $\sg$-semilinear bijection
$\Sg:M \rightarrow M$ \ie a bijection $\Sg$ such that
$\Sg(\lambda m)=\sg(\la)\Sg(m)$ for all $(\la,m) \in F \times M$.
\end{defn}

\par
One can attach a $\sg$-difference module $\cM_A:=(F^\nu,\Sg_A)$, with
$\Sg_A: F^\nu \rightarrow F^\nu$, $Y\mapsto A^{-1}\sg(Y)$,
to a $\sg$-difference system
\beq\label{eq:firstsys}
\sg(Y)=AY,\hbox{~with $A\in GL_\nu(F)$, for some $\nu\in\Z_{>0}$,}
\eeq
so that the horizontal
(\ie invariant) vectors
with respect to $\Sg_A$ correspond to the solutions of $\sg(Y)=AY$.
Conversely, the choice of an $F$-basis $\ul e$ of a difference module $\cM$
leads to a $\sg$-difference system $\sg(Y)=AY$, with $A\in GL_\nu(F)$,
which corresponds to the equation for the horizontal
vectors of $\cM$ with respect to $\Sg_A$ in the chosen basis.

\par
We define a \emph{morphism of ($\sg$-)difference modules over $F$}
to be an $F$-linear map between the underlying $F$-vector spaces,
commuting with the $\sg$-semilinear operators.
As defined above, the ($\sg$-)difference modules over $F$ form a Tannakian category
(see \cite{Delgrotfest}), \ie a category equivalent over the algebraic closure
$\overline{K}$ of $K$ to the category of finite-dimensional representations of an affine group scheme.
The affine group scheme corresponding to the sub-Tannakian category generated
by the($\sg$-)difference module $\cM_A$, whose non-Tannakian construction we are going to sketch below,
is called  Picard-Vessiot group of \eqref{eq:firstsys}.
Its structure measures the algebraic
relations satisfied by the solutions of  \eqref{eq:firstsys}.
\par
Let $\Delta:=\{ \partial_1,...,\partial_n\}$ be a set of commuting derivations of
$F$ such that, for all $i=1,...,n$, we have $\sg \circ \partial_i =\partial_i \circ \sg$.
In \cite{HardouinSinger}, the authors proved,
among other things,
that the category of $\sg$-modules carries also a
$\Delta$-structure \ie it is a differential
Tannakian category as defined by A. Ovchinnikov in \cite{ovchinnikovdifftannakian}.
The latter is equivalent to a category of finite-dimensional
representations of a differential  group scheme (see \cite{kolchindiffalg}),
whose structure measures the differential
relations
satisfied by the solutions of the $\sg$-difference modules.
In the next section, we describe the Picard-Vessiot approach to the theory in \cite{HardouinSinger}
(in opposition to the differential Tannakian approach), \ie the construction
of minimal rings containing the solutions  of $\sg(Y)=AY$ and their derivatives with respect to $\Delta$,
whose automorphism group is a concrete incarnation of the differential group scheme
defined by the differential
Tannakian equivalence.
We will implicitly consider the usual
Galois theory of ($\sg$-)difference equations by allowing $\De$
to be the empty set (see for instance \cite{vdPutSingerDifference}): we will
informally refer to this theory and the objects considered in it as classical.

\subsubsection{$\De$-Picard-Vessiot rings}
\label{subsubsec:DPV}

Let $F$ be a $(\sg,\De)$-field as above, with $K=F^\sg$.
Let us consider a $\sg$-difference system
\beq\label{eq:sys}
\sg(Y)=AY,
\eeq
with $A\in GL_\nu(F)$, as in \eqref{eq:firstsys}.

\begin{defn}[Def. 6.10 in \cite{HardouinSinger}]\label{defn:PV}
A $(\sg,\De)$-extension
$\cR$ of $F$ is a \emph{$\Delta$-Picard-Vessiot extension for \eqref{eq:sys}} if
\begin{enumerate}
\item
$\cR$ is a simple $(\sg,\De)$-ring \ie it has no nontrivial ideal stable under both $\sg$ and $\Delta$;

\item
$\cR$ is generated as a $\Delta$-ring by $Z \in GL_\nu(\cR)$ and $\frac{1}{det(Z)}$,
where $Z$ is a fundamental solution matrix of \eqref{eq:sys}.
\end{enumerate}
\end{defn}

\par
One can formally construct such an object as follows.
We consider the ring of rational $\Delta$-fucntions
$F\{X,\det X^{-1}\}_{\De}$.  We want to equip it with a structure of a $(\sg,\De)$-algebra, respecting the
commutativity conditions for $\sg$ and the $\partial_i$'s.
Therefore, we set $\sg(X)=A X$ and
\beq\label{eq:sigmastructure}
\begin{array}{rcl}
\sg( X^{\ul\a})
&=&\sg(\partial^{\ul\a}X)=\partial^{\ul\a}(\sg(X))=\partial^{\ul\a}(AX)\\~\\
&=&\ds\sum_{ \underline{i} +\underline{j}=\ul\a}
{\a_1\choose i_1}\cdots{\a_n\choose i_n}
\partial^{\underline{i}}(A) X^{\underline{j}},
\end{array}
\eeq
for each multi-index $\ul\a=(\a_1,\dots,\a_n)\in \N^n$.
Then the quotient $\cR$ of $F\{X,\det X^{-1}\}_{\De}$
by a maximal $(\sg,\De)$-ideal
obviously satisfies the conditions of the definition above and hence is
a $\De$-Picard-Vessiot ring. Moreover, it has the following properties:

\begin{prop}[Propositions 6.14 and 6.16 in \cite{HardouinSinger}]
If the field $K$ is $\Delta$-closed, then:
\begin{enumerate}
\item
The ring of constants $\cR^\sg$ of a $\Delta$-Picard-Vessiot ring $\cR$ for \eqref{eq:sys}
is equal to $K$, \ie there are no new constants with respect to $\sg$, compared to $F$.

\item
Two $\Delta$-Picard-Vessiot rings for \eqref{eq:sys} are isomorphic as $(\sg,\De)$-rings.
\end{enumerate}
\end{prop}

\subsubsection{$\De$-Picard-Vessiot groups}
\label{subsubsec:difPVgr}

\emph{Until the end of \S\ref{sec:rapdifgal},
\ie all along \S\ref{subsubsec:difPVgr} and
\S\ref{sec:difdependency}, we assume that $K$ is a $\Delta$-closed field.}

\medskip
Let $\cR$ be a $\Delta$-Picard-Vessiot ring for \eqref{eq:sys}.
Notice that, as in classical Galois theory for ($\sg$-)difference equations,
the ring $\cR$ does not need to be a domain.
One can show that it is in fact the direct sum of a finite number of copies of an integral domain, therefore
one can consider the ring $L$ of total fractions of $\cR$
which is isomorphic to the product of a finite number of copies of one field (see \cite{HardouinSinger}).

\begin{defn}\label{defn:pvgr}
The group
$Gal^{\Delta} (\cM_A)$ (also denoted $Aut^{\sg,\De}(L / F)$)
of the automorphisms of $L$ that fix $F$ and commute with $\sg$ and $\De$ is called the $\Delta$-Picard-Vessiot group
of \eqref{eq:sys}. We will also call it the Galois $\De$-group of \eqref{eq:sys}.
\end{defn}

\begin{rmk}
The group $Gal^{\De}(\cM_A)$
consists in the $K$-points of a linear algebraic $\De$-subgroup of
$GL_\nu(K)$ in the sense
of Kolchin. That is a subgroup of $GL_\nu(K)$ defined by a $\De$-ideal of $K\l\{X,\det X^{-1}\r\}_{\De}$.
\end{rmk}

Below, we recall some fundamental properties of the Galois $\De$-group, which are the starting
point of proving the Galois correspondence:

\begin{prop}[Lemma 6.19 in \cite{HardouinSinger}]~
\begin{enumerate}
\item The ring $L^{Gal^{\Delta} (\cM_A)}$ of elements of $L$ fixed by $Gal^{\Delta} (\cM_A)$ is $F$.
\item Let  $H$ be an algebraic $\Delta$-subgroup of $Gal^{\Delta}(\cM_A)$.
If $L^H = F$, then $H = Gal^{\Delta} (\cM_A)$.
\end{enumerate}
\end{prop}

As we have already pointed out, the Galois $\varnothing$-group is an algebraic
group defined over $K$ and corresponds to the classical Picard-Vessiot group
attached to the $\sg$-difference system
\eqref{eq:sys} (see \cite{vdPutSingerDifference,SauloyENS}):

\begin{prop}[Proposition 6.21 in \cite{HardouinSinger}]
The algebraic $\De$-group $Gal^{\Delta} (\cM_A)$ is a Zariski dense subset of $Gal^\varnothing(\cM_A)$.
\end{prop}

\subsubsection{Differential dependency and total integrability}
\label{sec:difdependency}

The $\Delta$-Picard-Vessiot ring $\cR$
of \eqref{eq:sys} is a $Gal^{\Delta} (\cM_A)$-torsor in the sense of Kolchin.
This implies in particular that the $\Delta$-relations satisfied
by a fundamental solution
of the $\sg$-difference system \eqref{eq:sys} are entirely determined by
$Gal^\De(\cM_A)$:

\begin{prop} [Proposition 6.29 in \cite{HardouinSinger}]\label{prop:degtrs}
The $\Delta$-transcendence degree of  $\cR$ over $F$ is equal to
the $\Delta$-dimension of $Gal(\cM_A)$.
\end{prop}

Since the $\Delta$-subgroups of ${\mathbb G}_a^n$ coincide with the zero set of
a homogeneous
linear  $\Delta$-polynomial $L(Y_1,...,Y_n)$ (see \cite{cassdiffgr}), we have:

\begin{prop}\label{prop:gatrans}
Let  $a_1, ..., a_n \in F$ and let $S$ be a $(\sg,\Delta)$-extension of $F$ such that $S^\sg =K$. If
$z_1,...,z_n \in S$  satisfy
$\sg (z_i) -z_i =a_i$ for $i=1,...,n$,
then $z_1,...,z_n \in S$  satisfy a nontrivial $\De$-relation
over $F$ if and only if there exists a nonzero
homogeneous linear differential polynomial $L(Y_1,...,Y_n)$ with coefficients in $K$ and
an element $ f \in F$ such that
$L(a_1, ..., a_n) =\sg (f)-f$.
\end{prop}

\begin{proof}
If $\De=\{\partial_1\}$, the proposition coincide with Proposition 3.1 in \cite{HardouinSinger}.
The proof in the case of many derivations is a straightforward generalization of their argument.
\end{proof}

The following proposition
relates the structural properties of the Galois $\De$-group
(see the remark immediately below for the definition
of constant $\De$-group) with the holonomy of a $\sg$-difference system:

\begin{prop}\label{prop:simpletrans}
The following statements are equivalent:
\begin{enumerate}

\item
The $\De$-Galois group $Gal^\Delta (\cM_A)$ is conjugate over $K$ to a constant $\De$-group.

\item
For all $i=1,...,n$, there exists a $B_i \in M_n(F)$ such that
the set of linear systems
$$
\left\{\begin{array}{l} \sg(Y)=AY \\
\partial_1 Y=B_1Y \\
\cdots\cdots\cdots\\
\partial_n Y=B_nY
\end{array} \right.
$$
is integrable, \ie the matrices $B_i$ and $A$ satisfy the
functional equations deduced from the commutativity of the operators:
$$
\partial_i(B_j) + B_jB_i =  \partial_j(B_i) + B_iB_j
\hbox{~and~}
\sg(B_j)A= \partial_j(A) +A B_j,
\hbox{~for any $i,j=1,\dots,n$.}
$$
\end{enumerate}
\end{prop}

\begin{proof}
The proof is a straightforward generalization of Proposition 2.9 in \cite{HardouinSinger} to the case of several derivations.
\end{proof}

\begin{rmk}\label{rmk:Deconstant}
Let $K$ be a $\De$-field and $C$ its subfield of $\De$-constants.
A linear $\De$-group $G \subset GL_\nu$ defined over $K$ is said to be
a constant $\De$-group (or $\De$-constant, for short) if one of the following equivalent statements hold:
\begin{itemize}

\item
the set of differential polynomials $\partial_h(X_{i,j})$, for $h=1,\dots,n$ and $i,j=1,\dots,\nu$,
belong to the ideal of definition of $G$
in the differential Hopf algebra $K\{X_{i,j},\frac{1}{\det\left(X_{i,j}\right)}\}_\De$ of $GL_\nu$ over $K$;

\item
the differential Hopf algebra of $G$ over $K$ is  an extension
of scalars of a finitely generated Hopf algebra over $C$;

\item
the points of $G$ in $K$ (which is $\De$-closed!) coincide the $C$-points of an algebraic group
defined over $C$.
\end{itemize}

For instance, let ${\mathbb G}_m$ be the multiplicative group defined over $K$. Its differential Hopf algebra
is $K\{x, \frac{1}{x}\}_\De$, \ie the $\De$-ring generated by $x$ and $\frac{1}{x}$.
The constant $\De$-group  ${\mathbb G}_m(C)$ corresponds to the differential Hopf algebra
$$
\frac{K\{x, \frac{1}{x}\}_\De}{(\partial_h(x); h=1,\dots,n)}
\cong C\l[x, \frac{1}{x}\r]\otimes_C K,
$$
where $C\l[x, \frac{1}{x}\r]$ is a $\De$-ring with the trivial action of the derivations in $\De$.
\end{rmk}

According to \cite{cassdiffgr}, if $H$ is a  Zariski dense $\De$-subgroup of a simple linear algebraic group $G\subset GL_\nu(K)$,
defined over a differentially closed field $K$,
then either $H=G$ or there exists $P\in GL_\nu(K)$ such that $PHP^{-1}$ is a constant $\De$-subgroup of $PGP^{-1}$.
Therefore:

\begin{cor}
If $Gal^\varnothing(\cM_A)$ is simple, we are either in the situation of the proposition above
or there are no $\De$-relations among the solutions of $\sg(Y)=AY$.
\end{cor}

\subsection{Descent over an algebraically closed field}
\label{part:desc}

\emph{In this subsection, we consider a $(\sg,\De)$-field  $F$, where $\sg$ is an automorphism of $F$ and
$\De=\{\partial\}$ is a set contaning only one derivation.
Moreover we suppose that $\sg$ commutes with $\partial$.}

\medskip
We have recalled above the theory developed in \cite{HardouinSinger}, where the authors assume that
the field of constants $K$ is differentially closed.
Although in most applications of \cite{HardouinSinger} a descent argument \emph{ad hoc}
proves that one can consider smaller, nondifferentially closed, field of constants,
the assumption that $K$ is $\Delta$-closed is quite restrictive.
We show here that if the $\sg$-constants $K$  of
$F$ form an algebraically closed field, we can construct a $\partial$-Picard-Vessiot ring,
whose ring of $\sg$-constants coincides with $K$, which allows us to descend the group introduced in the previous section
from a $\partial$-closed field to an algebraically closed field.
This kind of results were first tackled in \cite{peonnietodiffcomp}
using model theoretic arguments.
Here we use an idea of M. Wibmer of developing a differential analogue of Lemma 2.16 in \cite{Wibmchev}.
In \cite{wibmer2011existence}, M. Wibmer has given a more general version of Proposition \ref{prop:pvdesc}.
A Tannakian approach to the descent of parameterized Galois groups can be found in \cite{GilletGorchinskyOvchinnikov}.

\medskip
For now, we do not make any assumption on $K=F^\sg$.
Let $\Theta$ be the semigroup generated by $\partial$ and, for all $k \in \Z_{\geq 0}$, let $\Theta_{\leq k}$ be the
set of elements of  $\Theta$ of order less or equal to $k$.
We endow the differential rational function ring $\cS:=K\{X, \frac{1}{det(X)}\}_{\partial}$, where $X=(X_{i,j})$,
with the grading associated to the usual ranking
\ie we consider  for all $k \in\Z_{\geq 0}$ the rational function ring $\cS_k:= K\left[\be(X), \frac{1}{det(X)};\forall \be\in\Theta_{\leq k}\right]$.
Of course, we have $\partial(\cS_k) \subset \cS_{k+1}$.
It is convenient to set $\cS_{-1}=K$.

\begin{defn}[see \S3 in \cite{Landodiffker}]
The prolongation
of an ideal $\kI_k$ of $\cS_k$ is the ideal $\pi_1(\kI_k)$ of $\cS_{k+1}$ generated by $\Theta_{\leq 1}(\kI_k)$.
We say that a prime ideal $\kI_k$ of $\cS_k$, for $k\geq 0$, is a differential kernel of length $k$
if the prime ideal $\kI_{k-1}:=\kI_k \cap \cS_{k-1}$ of $\cS_{k-1}$ is such that
$\pi_1(\kI_{k-1} ) \subset \kI_k$.
\end{defn}

Notice that, according to the definition above, any prime ideal $\kI_0$ of $\cS_0$ is a differential kernel.

\begin{rmk}
In \cite{Landodiffker}, the author
defines a differential kernel
as a finitely generated field extension $F(A_0,A_1,...,A_{k}) /F$,
together with an extension of $\partial$
to a derivation of $F(A_0,A_1,...,A_{k-1})$ into $ F(A_0,A_1,...,A_k)$ such that $\partial (A_i) = A_{i+1}$ for $i=0,...,k-1$. We can recover $\kI_{k-1}$
as the kernel of the $F$-morphism
$\cS_k=F[ X,...,\partial^k(X) ] \to F(A_0,A_1,...,A_k)$, with $\partial^i(X)\mapsto A_i$.
We will not use Lando's point of view here.
\par
Notice that we have used Malgrange's prolongation $\pi_1$ (see \cite{MalgrangeInvolutifPS}) to rewrite Lando's definition.
One should pay attention to the fact tha Malgrange also considers a weak prolongation $\wtilde\pi_1$, which we won't consider here
(see Chapter V in \cite{MalgrangeInvolutifPS}).
\end{rmk}

\begin{prop}[Proposition 1 in \cite{Landodiffker}]\label{prop:diffker}
For $k\geq 0$, let $\kI_k$ be a  differential kernel  of $\cS_k$. There exists a differential kernel $\kI_{k+1}$ of $\cS_{k+1}$
such that  $\kI_k = \kI_{k+1} \cap \cS_k$. \end{prop}


Let $\sg(Y)=AY$ be  a  $\sg$-difference
system  with coefficient in $F$ as \eqref{eq:sys},
$\cR$ a $\partial$-Picard-Vessiot ring for $\sg(Y)=AY$,
constructed as in \S\ref{sec:rapdifgal} under the assumption that $K$ is $\partial$-closed, and
$\kI$ be the defining ideal of $\cR$, \ie the ideal such that
$$
\cR\cong\frac{F\left\{ X, \frac{1}{det(X)} \right\}_\partial}{\kI}.
$$
Then Proposition  6.21 in \cite{HardouinSinger} says that  $\kI_k:= \kI \cap \cS_k$ is a maximal $\sg$-ideal of $\cS_k$
endowed with the $\sg$-structure induced by $\sg(X)=AX$, which implies that
$\kI$ itself is a $\sg$-maximal ideal of $\cS$.
In order to prove the descent of the $\partial$-Picard-Vessiot ring $\cR$, we are going to proceed somehow in the opposite way.
Without any assumption on $K$, we will construct a sequence $(\kI_k)_{k \in \N}$ of $\sg$-maximal
ideals of $\cS_k$ such that $\cup_{k \in \N} \kI_k$ is a $\sg$-maximal ideal of $\cS$ stable by $\partial$. Such an ideal will provide us
with a $\partial$-Picard-Vessiot ring $\cR^\#$, which will be a
simple $\sg$-ring. If, moreover, $K$ is algebraically closed, we will be able to compare its group of automorphisms and the Galois $\partial$-group of the previous section.

\begin{prop}\label{prop:pvdesc}
Let $A\in GL_\nu(F)$. Then there  exists a $(\sg,\partial)$-extension $\cR^\#$ of $F$ such that:
    \begin{enumerate}
    \item $\cR^\#$ is generated over $F$ as a $\partial$-ring by $Z \in GL_\nu(\cR^\#)$ and $\frac{1}{det(Z)}$ for some matrix $Z$ satisfying $\sg(Z)=AZ$;
    \item $\cR^\#$ is a simple $\sg$-ring, \ie it has no nontrivial ideals stable under $\sg$.
    \end{enumerate}
\end{prop}

\begin{rmk}
Of course, a simple $\sg$-ring carrying a structure of $\partial$-ring is a simple $(\sg,\partial)$-ring and thus $\cR^\#$ is a $\partial$-Picard-Vessiot ring in the sense of Definition \ref{defn:PV}.
\end{rmk}

\begin{proof}
Let $\cS=F\left\{X, \frac{1}{det(X)}\right\}_{\partial}$ be the differential rational function ring in the variables $X=(X_{i,j})$ and
let $\cS_k$, $k\in\N$ be as above.
We define a $\sg$-ring structure on $\cS$ as in \eqref{eq:sigmastructure}, so that, in particular, $\sg(X)=AX$.
We will prove by induction on $k \geq 0$ that there exists a maximal
$\sg$-ideal $\kI_{k}$
of $\cS_{k}$ such that $\kI_k$ is a differential kernel of length $k$ and
$\kI_{k-1}=\kI_{k}\cap \cS_{k-1}$.
For $k=0$, we can take $\kI_0$ to be any $\sg$-maximal ideal of $\cS_0$. Then, the $\sg$-ring $\cS_0 / \kI_0$
is a classical Picard-Vessiot ring for $\sg(Y)=AY$ in the sense of \cite{vdPutSingerDifference}.
It follows from Proposition \ref{prop:diffker} that there exists a differential kernel $\kI_1$
of $\cS_1$ such that $\kI_1\cap\cS_0=\kI_0$ and $\pi_1(\kI_0)\subset\kI_1$.
\par
Now, let us construct $\kI_{k+1}$ starting from $\kI_{k-1}$ and $\kI_k$.
By Corollary 1.16 in \cite{vdPutSingerDifference},
both $\kI_{k-1}$ and $\kI_k$ can be written as intersections of the form:
$$
\kI_{k-1} =\bigcap_{i=0}^{t_{k-1}}\kJ_i^{(k-1)}
\l(\hbox{~resp.~}\kI_{k} = \bigcap_{i=0}^{t_{k}}\kJ_i^{(k)}\r),
$$
where the $\kJ_i^{(k-1)}$ (resp. $\kJ_i^{(k)}$) are prime ideals of $\cS_{k-1}$
(resp. $\cS_k$)\footnote{We point out, although the remark plays no role in the proof, that the Ritt-Raudenbush theorem on the $\partial$-Noetherianity of $\cS$
implies that the sequence of integers $(t_k)_{k \in \N}$ becomes stationary.}.
We shall assume that these representations are minimal and, so, unique.
Then,
\begin{enumerate}
\item
    the prime ideals $\kJ_i^{(k-1)}$ (resp. $\kJ_i^{(k)}$)
    are permuted by $\sg$;
\item
    for any $i=1,\dots,t_{k}$, there exists $j\in\{1,\dots,t_{k-1}\}$ such that $\kJ_i^{(k)}\cap \cS_{k-1}= \kJ_j^{(k-1)}$.
\end{enumerate}
The last assertion means that, for all $i=0,...,t_k$, the prime ideal $\kJ_i^{(k)}$ is a differential kernel of $\cS_k$.
Proposition \ref{prop:diffker} implies that
$\pi_1(\kJ_i^{(k)})$, and hence $\cap_{i=0}^{t_{k}}\pi_1(\kJ_i^{(k)})$ is a proper $\sg$-ideal of $\cS_{k+1}$.
Therefore there exists a  $\sg$-maximal ideal
$\kI_{k +1}$ of $\cS_{k+1}$ containing  $\cap_{i=0}^{t_{k}}\pi_1(\kJ_i^{(k)})$.
Moreover, the $\sg$-maximality of $\kI_k$
and the inclusion $\kI_k\subset\kI_{k+1} \cap \cS_k$ imply that $\kI_{k}=\kI_{k+1} \cap \cS_k$, which ends the
recursive argument.
The ideal $\kI =\cup_{k \in \N} \kI_k$ is clearly $\sg$-maximal in $\cS$ and $\partial$-stable. Then,
$\cR^\#:= \cS / \kI$ satisfies the requirements.
\end{proof}

\begin{rmk}
By Lemma 6.8 in \cite{HardouinSinger}, where the assumption that the field $K$ is $\partial$-closed plays no role,
we have that there exists a set of idempotents $e_0,\dots,e_r  \in \cR^\#$ such that $\cR^\# = \cR_0 \oplus \hdots \oplus \cR_r$,
where $\cR_i =e_i \cR^\#$ is an integral domain and
$\sg$ permutes the set $\{\cR_0, \dots,\cR_r \}$. Let $L_i$ be the fraction field of $\cR_i$. The total field of fraction $L^\#$ of $\cR^\#$
is equal to $L_0 \oplus \hdots \oplus L_r$.
\par
If $\wtilde K$ denotes a $\partial$-closure of $K$ and $\wtilde F$ is the fraction field of $\wtilde K \otimes_K F$ equipped
with an extension of $\sg$ that acts as the identity on $\wtilde K$, then
the construction of $\cR^\#$ over $\wtilde F$ gives a ring which is isomorphic to a
$\partial$-Picard-Vessiot extension of  the $\sigma$-difference system  $\sigma(Y)=AY$ viewed as a $\sigma$-difference system
with coefficients in $\wtilde F$, in the sense of \S\ref{subsubsec:DPV}.
\par
Notice that two Picard-Vessiot rings as in the proposition above may require a finitely generated extension of $K$ to become isomorphic (see \cite{wibmer2011existence}).
\end{rmk}

\begin{cor}
If $K =F^\sg$ is an algebraically closed field, the set of  $\sg$-constants  of $\cR^\#$ is equal to $K$.
\end{cor}

\begin{proof}
Let $c \in \cR^\#$ be a $\sg$-constant. In the notation of the previous proof,
there exists $k\in \N$ such that $ c \in \cR_k:=\cS_k/\kI_k$. By construction, $\cR_k$ is a simple $\sg$-ring and a finitely
generated $F$-algebra. By Lemma 1.8 in  \cite{vdPutSingerDifference}, the $\sg$-constants of $\cR_k$ coincide with $K$.
\end{proof}

\begin{prop} \label{prop:descgr}
Let $\sg(Y)=AY $ be a linear $\sg$-difference system with coefficients in $F$ and let $\cR^\#$
be the $\partial$-Picard-Vessiot ring constructed in Proposition \ref{prop:pvdesc}.
\par
If $K=F^\sg$ is algebraically closed, the  functor
\beq\label{eq:groupscheme}
\begin{array}{rccc}
Aut^{\sg,\partial}&:\hbox{$K$-$\partial$-algebras}&\longrightarrow&\hbox{Groups}\\
&S&\longmapsto&Aut^{\sg,\partial}(\cR^\#\otimes_K S / F\otimes_K S)
\end{array}
\eeq
is  representable by a linear algebraic $\partial$-group scheme $G_A$ defined over $K$.
Moreover, $G_A$ becomes isomorphic to $Gal^\partial( \cM_A)$ over a differential closure of $K$
(see Definition \ref{defn:pvgr} above).
\end{prop}

\begin{rmk}
Without getting into too many details, the representability of the functor \eqref{eq:groupscheme} is precisely
the  definition of a linear algebraic $\partial$-group scheme $G_A$ defined over $K$.
\end{rmk}

\begin{proof}
The first assertion is proved exactly as in \cite{HardouinSinger}, p. 368, where the authors only use the fact that the constants
of the $\partial$-Picard-Vessiot ring do not increase with respect to the base field $F$, the assumption that $K$ is
$\partial$-closed being used to prove this property of $\partial$-Picard-Vessiot rings.
The second assertions is just a consequence of the theory of differential
Tannakian category (see \cite{GilletGorchinskyOvchinnikov}),
which asserts that two differential fiber functors become isomorphic on a common $\partial$-closure of their fields of definition.
\end{proof}

\begin{defn}
We say that $G_A$ is the $\partial$-group scheme attached to $\sg(Y)=A Y$.
\end{defn}

Since we are working with schemes and not with the points of linear algebraic $\partial$-groups in a $\partial$-closure of $K$,
we need to consider functorial definitions of $\partial$-subgroup scheme and invariants (see \cite{MauGal} in the case of  iterative differential equations).
A $\partial$-subgroup functor $H$ of the functor $G_A $ is a $\partial$-group functor
$$
H:\{\mbox{$K$-$\partial$-algebras}\}  \rightarrow  \{\mbox{Groups}\}
$$
such that for all $K$-$\partial$-algebra $S$, the group $H(S)$ is a subgroup of $G_A(S)$.
So, let $L^\#$ be the total ring of fractions of $\cR^\#$ and let $H$ be a $\partial$-subgroup functor of $G_A$.
We say that $r=\frac{a}{b}\in L^\#$, with $a,b\in\cR^\#$, $b$ not a zero divisor, is an invariant of $H$ if for all $K$-$\partial$-algebra $S$
and all $h \in H(S)$, we have
$$
h(a \otimes 1). (b \otimes 1) = (a \otimes 1) . h(b\otimes 1).
$$
We denote the ring of invariant of $L^\#$ under the action of $H$ by $(L^\#)^H$.
\par
The Galois correspondence is proved by classical arguments starting from the following theorem.

\begin{thm}\label{thm:invariants}
Let $H$ be a $\partial$-subgroup functor $H$ of $ G_A$. Then $(L^\#)^H=K$ if and only if $H=G_A$.
\end{thm}

\begin{proof}
The proof relies on the same arguments as Theorem 11.4 in \cite{MauGal} and Lemma 6.19 in \cite{HardouinSinger}.
\end{proof}

\subsubsection{Descent of the criteria for hypertranscendency}
\label{subsubsec:descent}

The main consequence of Proposition \ref{prop:descgr} and Theorem \ref{thm:invariants} is that
Propositions \ref{prop:degtrs} and  \ref{prop:gatrans}  remain valid if one only assumes that the $\sg$-constants $K$ of the
base field
$F$ form an algebraically closed field and replaces the differential Picard-Vessiot group $Gal^\De(\cM_A)$ defined over
the differential closure of $K$ by the $\De$-group scheme $G_A$ defined over $K$.
In \cite{OvchinnikovGorchinskiy}, the authors exhibit a linear differential algebraic
group defined over $K$  which can be  conjugated  to a constant $\partial$-group   only over a transcendental extension of
$K$. Thus, one has to be careful when trying to prove a descended version of Proposition  \ref{prop:simpletrans}.
We go back to the multiple derivation case to state the following proposition, which is a generalization of
Corollary 3.12 in \cite{HardouinSinger}:

\begin{prop}\label{prop:intergrabilitydescent}
We suppose that $F$ is a $(\sg,\De)$-field\footnote{We recall that we assume that the derivation of $\De$ and $\sg$ commute
with each other.}, with $\De=\{\partial_1,\dots,\partial_n\}$,
such that $K=F^\sg$ is an algebraically closed field
and $F=K(x_1,\dots, x_n)$ a purely transcendental extension of transcendence basis
$x_1,\dots, x_n$. We suppose also that $\sg$ induces an automorphisms of $K(x_i)$ for any $i=1,\dots,n$.
\par
Let  $\wtilde K$  be  a $\De$-closure of $K$ and $\wtilde F$ be the fraction field of $\wtilde K \otimes_K F$
equipped with en extension of $\sg$ acting as the identity on $\wtilde K$,
so that one can consider $Gal^\De(\cM_A)$ as in \S\ref{subsubsec:difPVgr}.
\par
We consider a $\sigma$-difference system $\sigma(Y) =A Y$  with coefficients in $F$. The following statements are equivalent:
\begin{enumerate}
\item $Gal^\De(\cM_A)$ is conjugate over $\wtilde K$ to a constant $\De$-group.

\item There exist $B_1,\dots,B_n\in M_\nu (F)$ such that the system
$$
\left\{\begin{array}{l} \sg(Y)=AY \\
\partial_i Y=B_iY,~i=1,\dots,n\\
\end{array} \right.
$$
is integrable.
\end{enumerate}
\end{prop}

\begin{rmk}
Since $\sg$ induces an automorphism of $K(x_i)$ for any $i=1,\dots,n$, it acts on $x_i$ through a Moebius transformation.
By choosing another transcendence basis of $F/K$, we can suppose that
either $\sg(x_i)=q_ix$ for some $q_i\in K$ or $\sg(x_i)=x_i+h_i$ for some $h_i\in K$.
We are therefore in the most classical situation.
\end{rmk}

\begin{proof}
First of all we prove that the field $\wtilde F$ is a purely algebraic extension of $\wtilde K$.
Suppose that $\sg$ acts periodically on $x_1$, so that $r$ is the minimal positive integer such that $\sg^r(x_1)=x_1$. Then the polynomial
$$
(T-x_1)(T-\sg(x_1))\cdots(T-\sg^{r-1}(x_1))
$$
has coefficients in $K$ and vanishes at $x_1$. This is impossible since $x_1$ is transcendental over $K$.
Therefore we deduce that $\sg$ does not acts periodically on $x_1$, or on any
element of the transcendence basis.
\par
Suppose now that $x_1, \dots, x_n$, considered as elements of $\wtilde F$, satisfy
an algebraic relation over $\wtilde K$. This means that there exists a nonzero polynomial $P\in \wtilde K[T_1, \dots, T_n]$
such that $P(x_1,\dots,x_n)=0$. We can suppose that there exists $i_0=1,\dots,n$ such that $\sg(x_i)=q_ix$ for $i\leq i_0$
and $\sg(x_i)=x_i+h_i$ for $i>i_0$ (see remark above) and
choose $P$ so that the number of
monomials appearing in its expression is minimal. Let $T_1^{\a_1}\cdots T_n^{\a_n}$ be a monomial of maximal degree
appearing in $P$. We can of course suppose that its coefficient is equal to $1$.
Then $P(q_1T_1,\dots,q_{i_0}T_{i_0},T_{i_0+1}+h_{i_0},\dots,T_n+h_n)$ is another polynomial which annihilates at $x_1,\dots,x_n$.
It contains a monomial of higher degree of the form
$q_1^{\a_1}\cdots q_{i_0}^{\a_{i_0}}T_1^{\a_1}\cdots T_n^{\a_n}$. Therefore
$$
P(T_1, \dots, T_n)-q_1^{-\a_1}\cdots q_{i_0}^{-\a_{i_0}}P(q_1T_1,\dots,q_{i_0}T_{i_0},T_{i_0+1}+h_{i_0},\dots,T_n+h_n)
$$
also vanishes at $x_1,\dots,x_n$ and contains less terms than $P$, against our assumptions on $P$.
This proves that $x_1,\dots,x_n$ are algebraically independent over $\wtilde K$ and hence that $\wtilde F=\wtilde K(x_1,\dots,x_n)$.
\par
Now let us prove the proposition.
The group $Gal^\De(\cM_A)$ is the $\De$-Galois group of $\sigma(Y) =A Y$  considered as a $\sigma$-difference system with
coefficients in $\wtilde F$.
Then by Proposition  \ref{prop:simpletrans}, there exists $\wtilde B_i\in M_\nu (\wtilde F)$ such
that
\beq
\label{eq:compint}
\l\{\begin{array}{l}
\sigma(\wtilde B_i)=A\wtilde B_iA^{-1} + \partial_i(A)A^{-1}\\
\partial_i(\wtilde B_j)+\wtilde B_j\wtilde B_i=
\partial_j(\wtilde B_i)+\wtilde B_i\wtilde B_j
\end{array}\r..
\eeq
Now, we can replace the coefficients in $\wtilde B_i$ before the monomials in $\underline{x}$ by indeterminates.
From \eqref{eq:compint},
we obtain the system
$$
\l\{\begin{array}{l}
\sigma(\overline B_i)=A\overline B_iA^{-1} + \partial_i(A)A^{-1}\\
\partial_i(\overline B_j)+\overline B_j\overline B_i=
\partial_j(\overline B_i)+\overline B_i\overline B_j
\end{array}\r..
$$
with indeterminate coefficients, which
possess a specialization in $\wtilde K$. By clearing the denominators and identifying the coefficients of the monomials in the transcendence basis
$\underline{x}$, we see that this system is equivalent to a finite set of polynomial equations with coefficient in $K$. Since
$K$ is algebraically closed and this set of equations has a solution in $\wtilde K$, it must have a solution in $K$.
Thus, there exists $B_i\in M_\nu (F)$ with the required properties.
\par
On the other hand, if there exists $B_i\in M_\nu (F)$ with the required commutativity properties, then
it follows from Proposition
\ref{prop:simpletrans} that $Gal^\De(\cM_A)$ is  conjugate over $\wtilde K$ to
a constant $\De$-group.
\end{proof}

Going back to the one derivative situation, we have:

\begin{cor}\label{cor:intergrabilitydescent}
We suppose that $F/K$ a purely transcendental extension as in the previous proposition.
Let $\sigma(Y) =A Y$ be a $\sigma$-difference system with coefficients in $F$ and let $G_A$ be its $\partial$-group scheme over
$K$.
The following statements are equivalent:
\begin{enumerate}
\item $G_A$ is conjugate over $\wtilde K$ to a constant $\partial$-group.

\item There exists $B \in M_\nu (F)$ such that $\sigma(B)=ABA^{-1} + \partial (A)A^{-1} $,
\ie such that the system
$$
\left\{\begin{array}{l} \sg(Y)=AY \\
\partial Y=BY \\
\end{array} \right.
$$
is integrable.
\end{enumerate}
\end{cor}

\begin{proof}
Let us assume that $G_A$ is  conjugate over $\wtilde K$ to a constant $\partial$-group. Since $G_A$ becomes isomorphic to
$Gal^\partial(\cM_A)$ over $\wtilde K$, the latter is conjugate over $\wtilde K$ to
a constant $\partial$-group and we can conclude by the proposition above.
On the other hand, if there exists $B \in M_\nu (F)$ such that $\sigma(B)=ABA^{-1} + \partial (A)A^{-1}$, then
it follows from Proposition
\ref{prop:simpletrans} that $Gal^\partial(\cM_A)$ is  conjugate over $\wtilde K$ to
a constant $\partial$-group. Therefore the same holds for $G_A$.
So we can apply the previous proposition.
\end{proof}

\section{Confluence and $q$-dependency}

\subsection{Differential Galois theory for $q$-dependency}

Let $k$ be a characteristic zero field, $k(q)$ the field of rational functions in $q$ with coefficients in $k$ and $K$
a finite extension of $k(q)$.
We fix an extension $|~|$ to $K$ of the $q^{-1}$-adic valuation on $k(q)$.
This means that $|~|$ is defined on
$k[q]$ in the following way: there exists $d\in\R$, $d>1$, such that $|f(q)|=d^{\deg_q(f)}$ for any $f\in k[q]$. It extends by
multiplicativity to $k(q)$.
By definition,
we have $|q|>1$ and therefore it makes sense to consider elliptic functions with respect to $|~|$.
So let $(C, |~|)$ be the smallest valuated extension of $(K,|~|)$ which is both
complete and algebraically closed,
$\cM er(C^*)$ the field of meromorphic functions over $C^*:=C\smallsetminus\{0\}$ with respect to $|~|$, \ie the
field of fractions of the analytic functions over $C^*$, and
$C_E$ the field of elliptic functions on the torus $C^*/q^\Z$,
\ie the subfield of $\cM er(C^*)$ invariant with respect to the $q$-difference operators $\sgq:f(x)\mapsto f(qx)$.
\par
Since the derivation $\de_q=q\frac{d}{dq}$ is continuous on $k(q)$ with respect to $|~|$,
it naturally acts of the completion of $K$ with respect to $|~|$, and therefore on the completion
of its algebraic closure, which coincides with $C$ (see Chapter 3 in \cite{robert}).
It extends to
$\cM er(C^*)$ by setting $\de_q x=0$.
The fact that $\de_x=x\frac{d}{dx}$ acts on $\cM er(C^*)$ is straightforward.
We notice that
$$
\l\{\begin{array}{l}
\de_x\circ\sgq=\sgq\circ\de_x;\\
\de_q\circ\sgq=\sgq\circ(\de_x+\de_q).
\end{array}\r.
$$
This choice of the derivations is not optimal, in the sense that we would like to have two derivations commuting each other and,
more important, commuting to $\sgq$.
We are going to reduce to this assumption in two steps.
First of all, we consider the logarithmic derivative $\ell_q(x)=\frac{\de_x(\theta_q)}{\theta_q}$ of the Jacobi Theta function:
$$
\theta_q(x)=\sum_{n\in\Z}q^{-n(n-1)/2}x^n.
$$
We recall that, if $|q|>1$, the formal series $\theta_q$ naturally defines a meromorphic function on $C^*$ and satisfies the the $q$-difference equation
$$
\theta_q(qx)=qx\theta_q(x),
$$
so that $\ell_q(qx)=\ell_q(x)+1$.
This implies that $\sgq\de_x\l(\ell_q\r)=\de_x\l(\ell_q\r)$ and hence that $\de_x(\ell_q)$ is an elliptic
function.

\begin{lemma}\label{lemma:derivation1}
The following derivations of $\cM er(C^*)$
$$
\l\{
\begin{array}{l}
\delta_x,\\
\de=\ell_q(x)\delta_x+\delta_q.
\end{array}\r.
$$
commute with $\sgq$.
\end{lemma}

\begin{proof}
For $\de_x$, it is clear. For $\de$, we have:
$$
\begin{array}{rcl}
\de\circ\sgq(f(q,x))
&=& \l[\ell_q(x)\de_x+\de_q\r]\circ\sgq(f(q,x))\\ \\
&=&\sgq\circ\l[\l(\ell_q(q^{-1}x)+1\r)\de_x+\de_q\r]f(q,x)\\ \\
&=&\sgq\circ\l[\ell_q(x)\de_x+\de_q\r]f(q,x)\\ \\
&=&\sgq\circ\de(f(q,x)).
\end{array}
$$
\end{proof}

\begin{cor}~
\begin{enumerate}
\item
The derivations $\de_x$, $\de$ of $\cM er(C^*)$ stabilize $C_E$ in $\cM er(C^*)$.
\item
The field of constants $\cM er(C^*)^{\de_x,\de}$ of $\cM er(C^*)$ with respect to
$\de_x$, $\de$ is equal to the algebraic closure $\ol k$ of $k$ in $C$.
\end{enumerate}
\end{cor}

\begin{proof}
The first part of the proof immediately follows from the lemma above.
The constants of $\cM er(C^*)$ with respect to $\de_x$ coincide with $C$.
As far as the constants of $C$ with respect to $\de$ is concerned, we are reduced to determining
the constants $C^{\de_q}$ of $C$ with respect to $\de_q$.
Since the topology induced by $|~|$ on $k$ is trivial, one concludes that $C^{\de_q}$
is the algebraic closure of $k$ in $C$.
\end{proof}

Since
$$
[\de_x,\de]=\de_x\circ\de-\de\circ\de_x=\de_x(\ell_q(x))\de_x,
$$
we can consider a $(\de_x,\de)$-closure $\wtilde C_E$ of $C_E$ (see \cite{yaffenoncommder}, \cite{piercenoncommder} and \cite{singernoncommder}).
We extend $\sgq$ to the identity of $\wtilde C_E$.
The $(\de_x,\de)$-field of $\sgq$-constants $\wtilde C_E$
almost satisfies the hypothesis of \cite{HardouinSinger}, apart from the fact that $\de_x$, $\de$ do not commute with each other:

\begin{lemma}\label{lemma:derivation2}
There exists $h\in\wtilde C_E$ verifying the differential equation
$$
\de(h)=\de_x(\ell_q(x))h,
$$
such that the derivations
$$
\l\{\begin{array}{l}
\partial_1=h\de_x;\\
\partial_2=\de=\ell_q(x)\de_x+\de_q.
\end{array}\r.
$$
commute with each other and with $\sgq$.
\end{lemma}

\begin{proof}
Since $\ell_q(qx)=\ell_q(x)+1$, we have $\de_x(\ell_q(x))\in C_E$.
Therefore, we are looking for a solution $h$ of a linear differential
equation of order $1$, with coefficients in $C_E$.
Let us suppose that $h\in\wtilde C_E$ exists. Then since $\sgq h=h$, the identity
$\sgq\circ\partial_i=\partial_i\circ\sgq$ follows from Lemma \ref{lemma:derivation1} for $i=1,2$.
The verification of the fact that $\partial_1\circ\partial_2=\partial_2\circ\partial_1$ is straightforward and, therefore, left to the reader.
\par
We now prove the existence of $h$. Consider the differential rational function ring
$\cS:= C_{E}\{ y, \frac{1}{y} \}_{\de_x}$.
We endow $\cS$ with an extension of $\de$ as follows:
$$
\de(\de_x^n(y))= \de_{x}^{n+1}(\ell_{q}) y
$$
for all $n \geq 0$.
Since $\de_x\circ\de=\de_x(\ell_q(x))\de_x+\de\circ \de_x$, the definition of $\de$ over $\cS$ is consistent
and the commutativity relation between $\de_x$ and $\de$ extends from $C_E$ to $\cS$.
Now let $\mathfrak{M}$ be a maximal $(\de_x,\de)$-maximal ideal of $\cS$. Then, the
ring $\cS / \mathfrak{M}$ is a simple $(\de_x,\de)$-$C_{E}$-algebra. By Lemma 1.17 in \cite{vdPutSingerDifferential}, it is also an integral domain. Let $L$ be
the quotient field of $\cS / \mathfrak{M}$. The field $L$ is a $(\de_x,\de)$-field extension of $C_{E}$ which contains
a solution of the equation $\de(y)= \partial_x(\ell_q(x))y$.
Since $\wtilde C_E$ is the $(\de_x,\de)$-closure of $C_{E}$, there exists $h\in\wtilde C_E$ satisfying the differential equation
$\de(h)=\de_x(\ell_q(x))h$.
\end{proof}

Let $\De=\{\partial_1,\partial_2\}$.
Notice that, since $\ell_q(qx)=\ell_q(x)+1$ and $\sgq$ commutes with $\De$,
we have:
$$
\sgq\partial_1 (\ell_q(x))=\partial_1 (\ell_q(x)),\,\,\sgq(\partial_{2} (\ell_q(x))) =\partial_{2} (\ell_q(x)),
$$
and therefore $\partial_i(\ell_q(x))\in C_E$ for $i=1,2$.
We conclude that the subfield $C_{E}(x,\ell_q(x))$
of $\cM er(C^*)$ is actually a $(\sgq,\De)$-field.
Moreover, extending the action of $\sgq$ trivially to $\wtilde C_E$,
we can consider the $(\sgq,\De)$-field $\wtilde{C}_{E}(x,\ell_q(x))$.
Since the fields $C_E(x,\ell_q)$ and $\wtilde C_E$ are linearly disjoint over $C_{E}$
(see Lemma 6.11 in \cite{HardouinSinger}), $\wtilde C_E(x,\ell_q(x))$ has a $\De$-closed field of constants, which coincide with $\wtilde C_E$.

\subsection{Galois $\De$-group and $q$-dependency}

The subsection above shows that one can attach two linear differential algebraic groups to a $q$-difference system $\sg_q(Y)=A(x) Y$ with $A \in GL_n(C(x))$:
\begin{enumerate}
\item
The group $Gal^\De(\cM_A)$ which corresponds to Definition \ref{defn:pvgr} applied to the $(\sgq,\De)$-field
$\wtilde C_E(x,\ell_q)$.
This group is defined over  of $\wtilde C_E$ and
measures all differential relations satisfied by  the solutions of the $q$-difference equation with respect to $\delta_x$ and $\delta_q$. However its computation may be a little difficult. Indeed,
since the derivations of $\Delta$ are themselves defined above $\wtilde C_E$,
there is no hope of a general descent argument.
Nonetheless, in some special cases, one can use the linear disjunction
of the field $\wtilde C_E^\De$ of $\De$-constants and $C_E(x,\ell_q)$ above $C_E$ to simplify the computations.

\item
The Galois $\partial_2$-group
$Gal^{\partial_2}(\cM_A)$ which corresponds to Definition \ref{defn:pvgr} applied to
the $(\sgq,\partial_2)$-field $\wtilde C_E(x,\ell_q)$.
\end{enumerate}

Let us consider the $q$-difference system
\beq \label{eq:sys2}
Y(qx) =A(x) Y(x),
\hbox{~with $A\in GL_\nu(C_E(x, \ell_q))$.}
\eeq
In view of Proposition \ref{prop:descgr}, the Galois $\partial_2$-group
$Gal^{\partial_2}(\cM_A)$ attached to \eqref{eq:sys2} is  defined
above  the algebraic closure $\overline C_E$ of $C_E$.
We will prove below that, in fact, it descends to $C_E$ and thus reduce all the computations to calculus over the field of elliptic functions.
\par
The field $C_E(x, \ell_q)$ is a subfield of the field of meromorphic functions over $C^*$,
therefore \eqref{eq:sys2} has a fundamental solution matrix $U\in GL_\nu(\cM er (C^*))$. In fact,
the existence of such a fundamental solution $U$ is actually equivalent
to the triviality of the pull-back on $C^*$ of vector bundles over the torus $C^*/q^\Z$
(see \cite{Praag2} for an analytic argument).
The $(\sgq,\partial_2)$-ring $\cR_{\cM er}:=C_E(x, \ell_q)\{U,\det U^{-1}\}_{\partial_2}\subset \cM er(C^*)$
is generated as a $(\sgq,\partial_2)$-ring by a fundamental solution $U$ of \eqref{eq:sys2} and by $\det U^{-1}$ and has the property that
$\cR_{\cM er}^{\sgq}=C_E$, in fact:
$$
C_E\subset\cR^{\sgq}\subset\cM er( C^*)^{\sgq}=C_E.
$$
Notice that $\cR_{\cM er}$ does not need to be a simple $(\sgq,\partial_2)$-ring.
For this reason we call it a  weak $\partial_2$-Picard-Vessiot ring.
We have:

\begin{prop}
$Aut^{\sgq,\partial_2}(\cR_{\cM er}/  C_E(x, \ell_q))$ consists of the $C_E$-points of a linear algebraic $\partial_2$-group $G_{C_E}$ defined over $C_E$ such that
$G_{C_E}\otimes_{C_E}\overline C_E\cong Gal^{\partial_2}(\cM_A)$.
\end{prop}

\begin{proof}
See the proof of
Theorem 9.5 in \cite{diviziohardouin}, which gives an analogous statement for the derivation $x\frac{d}{dx}$.
\end{proof}

Therefore, as in Proposition \ref{prop:gatrans}, one can prove:

\begin{cor}\label{cor:gatransCE}
Let  $a_1, ..., a_n \in C_E(x,\ell_q)$ and let $S\subset\cM er(C^*)$ be a $(\sgq,\partial_2)$-extension of $C_E(x,\ell_q)$ such that $S^{\sgq} =C_E$. If
$z_1,...,z_n \in S$  satisfy
$\sg (z_i) -z_i =a_i$ for $i=1,...,n$,
then $z_1,...,z_n \in S$  satisfy a nontrivial $\partial_2$-relation
over $C_E(x,\ell_q)$ if and only if there exists a nonzero
homogeneous linear differential polynomial $L(Y_1,...,Y_n)$ with coefficients in $C_E$ and
an element $ f \in C_E(x,\ell_q)$ such that
$L(a_1, ..., a_n) =\sg (f)-f$.
\end{cor}

\begin{proof}
The proof strictly follows Proposition 3.1 in \cite{HardouinSinger}, so we only sketch it.
First of all, if such an $L$ exists, $L(z_1,...,z_n)-f$ is $\sgq$-invariant and therefore belongs to
$C_E$, which gives a nontrivial $\partial_2$-relation among the $z_1,...,z_n$.
On the other hand, we can suppose that $S$ coincides with the ring $\cR_{\cM er}$ introduced above for the system
$\sgq Y=AY$, where $A$ is the diagonal block matrix
$$
A=diag\l(\begin{pmatrix}1&a_1\\0&1\end{pmatrix},\dots,
\begin{pmatrix}1&a_n\\0&1\end{pmatrix}\r).
$$
In fact, a solution matrix of $\sgq Y=AY$ is given by
$$
U=diag\l(\begin{pmatrix}1&z_1\\0&1\end{pmatrix},\dots,
\begin{pmatrix}1&z_n\\0&1\end{pmatrix}\r)\in  GL_{2n}(\cM er(C^*)).
$$
This implies that $Aut^{\sgq,\partial_2}(\cR_{\cM er}/  C_E(x, \ell_q))$ is a $\partial_2$-subgroup of the group $(C_E,+)^n$,
therefore there exists a nonzero
homogeneous linear differential polynomial $L(Y_1,...,Y_n)$ with coefficients in $C_E$ such that
$Aut^{\sgq,\partial_2}(\cR_{\cM er}/  C_E(x, \ell_q))$ is contained in the set of zeros of $L$ in $(C_E,+)^n$.
We set $f=L(z_1,\dots,z_n)$.
A Galoisian argument shows that $f\in C_E(x,\ell_q)$ and that $L(a_1,\dots,a_n)=\sgq(f)-f$.
\end{proof}

\subsection{Galoisan approach to heat equation}

We want to show how the computation of the Galois $\partial_2$-group of the $q$-difference equation $y(qx)=qx y(x)$ leads to the heat equation.
We recall that the Jacobi theta function verifies $\theta_q(qx) =qx \theta_q(x)$.
Corollary \ref{cor:gatransCE} applied to this equation becomes: the function $\theta_q$ satisfies a $\partial_2$-relation
with coefficients in $C_E(x,\ell_q)$ if and only if there exist $a_1,...,a_m\in C_E$ and $f \in C_E(x,\ell_q)$ such that
$$
\sum_{i=0}^m a_i \partial_2^{i} \l(\frac{\partial_2(qx)}{qx}\r) = \sgq(f) -f.
$$
A simple computation  leads to
$$
\frac{\partial_2(qx)}{qx} =\ell_q+1 = \sgq\l(\frac{1}{2} \l(\ell_q^2 + \ell_q\r) \r)- \l(\frac{1}{2} \l(\ell_q^2 +\ell_q\r)\r)
$$
and therefore to
$$
\sgq\l( 2 \frac{\partial_2(\theta_q)}{\theta_q} -\l(\ell_q^2 + \ell_q\r)\r) =2 \frac{\partial_2(\theta_q)}{\theta_q} -(\ell_q^2+\ell_q).
$$
The last identity is equivalent to the fact that
\beq\label{eq:identityheat}
2 \frac{\partial_2(\theta_q)}{\theta_q} - (\ell_q^2 + \ell_q) =  2\frac{\delta_q(\theta_q)}{\theta_q} + \ell_q^2 - \ell_q
\eeq
is an elliptic function and
implies that
\beq\label{eq:groupheat}
Gal^{\partial_2}(\cM_{qx}) \subset\l\{ \frac{\partial_2(c)}{c} =0\r\}.
\eeq
The heat equation\footnote{One deduces from \eqref{eq:heat} that
$\wtilde\theta(q,x):=\theta_q(q^{1/2}x)$ satisfies the
equation $2\de_q\wtilde\theta=\de_x^2\wtilde\theta$.
Then, to recover the classical form of the heat equation, it is enough make the change of variables $q=\exp(-2i\pi\tau)$ and $x=\exp(2i \pi z)$.}
\beq\label{eq:heat}
2\de_q\theta_q=-\de_x^2\theta_q+\de_x\theta_q
\eeq
can be rewritten as
$$
2 \frac{\de_{q}\theta_q}{\theta_{q}}+\ell_{q}^2 -\ell_{q}= - \de_{x}(\ell_{q}).
$$
Since $\de_{x}(\ell_q) $ is an elliptic function,
taking into account \eqref{eq:identityheat}, we see
that the Galois $\partial_2$-group of $\theta_q$, as described in \eqref{eq:groupheat}, can somehow be viewed as
a Galoisian counterpart of the heat equation.

\subsection{q-hypertranscendency of rank 1 $q$-difference equations}

We want to study the $q$-dependency of the solutions of a $q$-difference equation of the form $y(qx)=a(x) y(x)$, where $a(x)\in k(q,x)$,
$a(x)\neq 0$.

\begin{thm}\label{thm:qconfluencerank1}
Let $u$ be a  nonzero meromorphic solution of $y(qx)=a(x) y(x)$, $a(x)\in k(q,x)$, in the sense of the previous subsections.
The following statements are equivalent:
\begin{enumerate}

\item
$a(x) = \mu  x^r \frac{g(qx)}{g(x)}$ for some $r \in \Z$, $g \in k(q,x)$ and $\mu \in k(q)$.

\item
$u$ is a solution of a nontrivial algebraic $\de_x$-relation with coefficients in
$C_E(x,\ell_q)$ (and therefore in $C(x)$).

\item
$u$ is a solution of a nontrivial algebraic $\partial_2$-relation with coefficients in
$C_E(x,\ell_q)$.

\end{enumerate}
\end{thm}

First of all, we remark that the equivalence between 1. and 2. follows from Theorem 1.1 in \cite{hardouincompositio}
(replacing $\C$ by $k(q)$ is not an obstacle in the proof). Moreover being $\de_x$-algebraic over $C_E(x,\ell_q)$ or over
$C(x)$ is equivalent, since $C_E(\ell_q)$ is $\de_x$-algebraic over $C(x)$.
The equivalence between 1. and 3. is proved in the proposition below:

\begin{prop}
Let $u$ be a  nonzero meromorphic solution of $y(qx)=a(x) y(x)$, with $a(x)\in k(q,x)$.
Then $u$ satisfies an algebraic differential equation with respect to $\partial_2$ with coefficients
in $C_E(x,\ell_q)$ if and only if $a(x) = \mu  x^r \frac{g(qx)}{g(x)}$ for some $r \in \Z$, $g \in k(q,x)$ and $\mu \in k(q)$.
\end{prop}

\begin{proof}
By Lemma 3.3 in \cite{hardouincompositio}, there exists $f(x)\in k(q,x)$ such that $a(x) =\wtilde a(x)\frac{f(qx)}{f(x)}$ and
$\wtilde a(x)$ has the property that if $\a$ is a zero (resp. a pole) of $\wtilde a(x)$, then
$q^n\a$ is neither a zero nor a pole of $\wtilde a(x)$ for any $n\in\Z\smallsetminus\{0\}$.
Replacing $u$ by $\frac{u}{f(x)}$, we can suppose that $a(x)=\wtilde a(x)$ and we can write $a(x)$
in the form:
$$
a(x) =\mu x^r \prod_{i=1}^s (x- \alpha_i)^{l_i},
$$
where $\mu \in k(q)$, $r\in\Z$, $l_1,...,l_s \in \Z$
and, for all $i=1,...,s$, the $\alpha_i$'s are nonzero elements of a fixed algebraic closure of $k(q)$, such that $q^\Z\a_i\cap q^\Z\a_j=\varnothing$ if $i\neq j$.
By Corollary \ref{cor:gatransCE}, the solutions of $y(qx)=a(x) y(x)$ will satisfy a nontrivial
algebraic differential equation in $\partial_2$ if and only if there exists $f \in C_E(x,\ell_q)$, $a_1,...,a_m\in C_E$
such that
$$
\sum_{i=0}^m a_i \partial_2^{i}\l(\frac{\partial_2(a(x))}{a(x)}\r) =f(qx)-f(x).
$$
We can suppose that $a_m=1$.
Notice that
$$
\frac{\partial_2(a(x))}{a(x)}
=\frac{\partial_2( x^r)}{x^r}
            +\frac{\delta_q(\mu)}{\mu}
            +\sum_{i=1}^{s} \frac{ \partial_2(x-\a_i)^{l_i}}{(x -\alpha_i)^{l_i}},
$$
where
$$
\frac{\partial_2( x^r)}{x^r} =r \ell_q(x)= \l(\frac{r}{2} (\ell_q^2-\ell_q)\r)(qx) - \l(\frac{r}{2} (\ell_q^2-\ell_q)\r) (x),
\hbox{~with $\ell_q^2-\ell_q\in C_E(x,\ell_q)$,}
$$
and
$$
\frac{\delta_q(\mu)}{\mu}=\l(\frac{\delta_q(\mu)}{2\mu}\ell_q\r)(qx)-\l(\frac{\delta_q(\mu)}{2\mu}\ell_q\r)(x),
\hbox{~with $\frac{\delta_q(\mu)}{2\mu}  \ell_q\in C_E(x,\ell_q)$.}
$$
It remains to show that a solution of  $y(qx)=a(x) y(x)$ satisfies a nontrivial
differential equation in $\partial_2$ if and only if there exists $h \in C_E(x,\ell_q)$  such that
\beq \label{eq:qdep}
\sum_{j=0}^m a_j \partial_2^{j} \left( \sum_{i=1}^{s} \frac{ l_i( x\ell_q(x) -\delta_q(\alpha_i))}{(x -\alpha_i)}\right) =h(qx)-h(x).
\eeq
If we prove that \eqref{eq:qdep} never holds,
we can conclude that a solution of  $y(qx)=a(x) y(x)$ satisfies a nontrivial
algebraic differential equation in $\partial_2$ with coefficients in $C_E(x,\ell_q)$
if and only if $a(x) = \mu  x^r$ (modulo the reduction done at the beginning of the proof).
For all $i=1,\dots,s$ and $j=0,\dots,m$, the fact that $\partial_2\ell_q(x)\in C_E$ allows to prove inductively that:
\beq \label{eq:derivexpr}
\partial_2^{j}
\left(\frac{l_i( x\ell_q(x) -\delta_q(\alpha_i))}{(x -\alpha_i)}\right)
= \frac{l_i (-1)^j j!( x\ell_q(x) -\delta_q(\alpha_i))^{j+1}}{(x -\alpha_i)^{j+1}} +\frac{h_{i,j}}{(x-\alpha_i)^j},
\eeq
for some $h_{i,j} \in C_E[ \ell_q]$.
\par
Since  $x$ is transcendental over $C_E(\ell_q)$, we  can consider $ f(x)=
\sum_{j=0}^m a_j \partial_2^{j} \left( \sum_{i=1}^{s} \frac{ l_i( x\ell_q(x) -\delta_q(\alpha_i))}{(x -\alpha_i)}\right) $
as a rational function in $x$ with coefficients in $C_E(\ell_q)$.
In the partial fraction decomposition of $f(x)$, the polar term in $\frac{1}{(x -\alpha_i)}$ of
highest order is $ \frac{l_i (-1)^{m} (m!)( x\ell_q(x) -\delta_q(\alpha_i))^{m+1}}{(x -\alpha_i)^{m+1}}$.
By the partial fraction decomposition theorem, identities \eqref{eq:qdep} and \eqref{eq:derivexpr} imply that
this last term appears
either in the decomposition of  $h(x)$ or in the decomposition of $h(qx)$.
In both cases, there exists $s \in \Z^*$ such that the  term $\frac{1}{(q^s x -\alpha_i)^{m+1}}$ appears in the partial
fraction decomposition of $h(qx)-h(x)$. This is in contradiction
with the assumption that the poles $\alpha_i$ of $a(x)$ satisfy  $q^\Z\a_i\cap q^\Z\a_j=\varnothing$ if $i\neq j$.
\end{proof}

\begin{rmk}
The Jacobi theta function is an illustration of the theorem above.
Notice that a meromorphic solution of the $q$-difference equations of the form $y(qx)=a(x)y(x)$, with $a(x) = \mu  x^r \frac{g(qx)}{g(x)}\in k(q,x)$
is given by $y(x)=\frac{\theta_q(\mu x/q^r)}{\theta_q(x)}\theta(x)^rg(x)\in\cM er (C^*)$.
\end{rmk}

\subsection{Integrability in $x$ and $q$}

Let $Y(qx)=AY(x)$ be a $q$-difference system with $A\in GL_\nu(C(x,\ell_q))$ and $\De=\{\partial_1,\partial_2\}$ as
in Lemma \ref{lemma:derivation2}. We deduce the following from Proposition \ref{prop:intergrabilitydescent}.

\begin{cor}\label{cor:qconf}
The Galois $\De$-group $Gal^\De(\cM_A)$ of $Y(qx)=AY(x)$ is $\De$-constant (see Remark \ref{rmk:Deconstant})
if and only if there exist two square matrices $B_1,B_2 \in M_\nu(\overline C_E(x,\ell_q))$ such that
the system
$$
\l\{
\begin{array}{l}
\sgq(Y)=AY(x)\\
\de_x Y=B_1 Y\\
\partial_2 Y=B_2Y
\end{array}\r.
$$
is integrable, in the sense that these matrices verify:
\begin{eqnarray}
&\sgq(B_1)= AB_1A^{-1} + \de_x(A).A^{-1},\label{sgqdex}\\
&\sgq(B_2)= AB_2A^{-1} + \partial_2(A).A^{-1},\label{sgqpartial2}\\
&\partial_2(B_1) + B_1B_2 + \de_x(\ell_q) B_1 =  \de_x(B_2) + B_2B_1.\label{dexpartial2}
\end{eqnarray}
\end{cor}

\begin{proof}
It follows from Proposition \ref{prop:simpletrans} that there exist
$B_1, B_2 \in M_\nu(\wtilde C_E(x,\ell_q))$
such that the system
$$
\l\{
\begin{array}{l}
\sgq(Y)=AY(x)\\
\partial_1 Y=hB_1 Y\\
\partial_2 Y=B_2Y
\end{array}\r.
$$
is integrable. This is easily seen to be equivalent to
\eqref{sgqdex}, \eqref{sgqpartial2} and \eqref{dexpartial2}.
Notice that, as in Proposition \ref{prop:intergrabilitydescent}, we have:
\begin{itemize}
\item
$x$ and $\ell_q$ are transcendental and algebraically independent over $\ol C_E$ (see for instance Theorem 3.6 in
\cite{hardouincompositio});
\item
$\sgq(x)=qx$ and $\sgq(\ell_q)=\ell_q+1$.
\end{itemize}
It follows that $\ol C_E(x,\ell_q)\otimes_{\ol C_E}\wtilde C_E=\wtilde C_E(x,\ell_q)$ and that, therefore,
the same argument as in the proof of Proposition \ref{prop:intergrabilitydescent} allows to conclude.
\end{proof}

\begin{exa}
We want to study the $q$-dependency of the $q$-difference system
\beq
\label{eq:diagconf} Y(qx) =\begin{pmatrix} \la &  \eta \\
0 & \la \end{pmatrix} Y(x),
\eeq
where $\la, \eta \in k(q,x)$.
First of all, the solutions of the equation $y(qx)=\la y(x)$ will admit a $\De$-relation if and only if
$\la=\mu x^r\frac{g(qx)}{g(x)}$ for some $r\in\Z$, $g\in k(q,x)$ and $\mu\in k(q)$ (see Theorem \ref{thm:qconfluencerank1}).
To simplify the exposition we can suppose that $\la=\mu x^{r}$.
Then clearly the Galois $\De$-group of $y(qx)=\la y(x)$ is $\De$-constant if and only if
$\mu\in k(q)$ is a power of $q$, \ie if and only if $\mu\in q^\Z$. Let $\mu=q^s$ for $s\in\Z$.
A solution of $y(qx)=\la y(x)$ is given by
$$
y(x)=x^{s-r}\theta_q(x)^r.
$$
To obtain an integrable system
$$
y(qx)=\mu x^r y(x),~\de_x y=b_1 y,~\partial_2 y=b_2 y,
$$
satisfying \eqref{sgqdex}, \eqref{sgqpartial2} and \eqref{dexpartial2}, it is enough to take:
$$
b_1=\frac{\de_x(y)}{y},~b_2=\frac{\partial_2(y)}{y}.
$$
If $\mu=q^s\in q^\Z$ then $b_1=r\ell_q(x)+s-r\in C_E(x,\ell_q(x))$.
The same hypothesis on $\mu$ implies that $b_2\in C_E(x,\ell_q(x))$, in fact we have:
$$
\sgq\l(\frac{\partial_2(\theta_q(x))}{\theta_q(x)}\r)=(\ell_q(x)+1)+\frac{\partial_2(\theta_q(x))}{\theta_q(x)},
$$
which implies that
$$
\frac{\partial_2(\theta_q(x))}{\theta_q(x)}=\frac{\ell_q(x)(\ell_q(x)+1)}{2}+e(x),
\hbox{~for some $e(x)\in C_E$.}
$$
To go back to the initial system, we have to find
$$
B_1=\begin{pmatrix}b_1 & \a(x)\\0&b_1\end{pmatrix},
B_2=\begin{pmatrix}b_2 & \be(x)\\0&b_2\end{pmatrix}
\in M_2(\ol C_E(x,\ell_q)),
$$
satisfying
\eqref{sgqdex}, \eqref{sgqpartial2} and \eqref{dexpartial2}.
We find that
the $q$-difference system \eqref{eq:diagconf} has constant Galois $\De$-group if and only if there exist $\a(x),\be(x)\in \overline C_E(x,\ell_q)$ such that
\begin{itemize}
\item $\de_x (\eta) = r \eta  +(\a(qx)-\a(x)) \mu x^{r}$;
\item $\de_q( \eta) = \frac{\de_q \mu}{\mu} \eta + \left( \be(qx)-\be(x) + \ell_q (\a(qx) - \a(x)) \right)\mu x^{r}$.
\item $\partial_2(\a(x)) + \de_x(\ell_q) \a(x)= \de_x (\be(x))$.
\end{itemize}
\end{exa}

\begin{rmk}
In \cite{PulitaComp}, Pulita shows that given a $p$-adic differential equations $\frac{dY}{dx}=GY$ with coefficients in some classical algebras of functions,
or a $q$-difference equations $\sgq(Y)=AY$, one can always complete it in an integrable system:
$$
\l\{
\begin{array}{l}
\sgq(Y)=AY\\
\frac{dY}{dx}=GY\\
\frac{dY}{dq}=0
\end{array}
\r..
$$
This is not the case in the complex framework, even extending the coefficients from $C_E(\ell_q,x)$ to
$\cM er(C^*)$.
\end{rmk}


\begin{thebibliography}{DVH10b}

\bibitem[Cas72]{cassdiffgr}
P.J. Cassidy.
\newblock Differential algebraic groups.
\newblock {\em American Journal of Mathematics}, 94:891--954, 1972.

\bibitem[Coh65a]{Cohn:difference}
Richard~M. Cohn.
\newblock {\em Difference algebra}.
\newblock Interscience Publishers John Wiley \& Sons, New York-London-Sydeny,
  1965.

\bibitem[Coh65b]{Cohn}
R.M. Cohn.
\newblock {\em Difference algebra}.
\newblock Interscience Publishers John Wiley \& Sons, New York-London-Sydeny,
  1965.

\bibitem[Del90]{Delgrotfest}
P.~Deligne.
\newblock Cat{\'e}gories tannakiennes.
\newblock In {\em in The Grothendieck Festschrift, Vol II}, volume~87 of {\em
  Prog.Math.}, pages 111--195. Birkha{\"a}user, Boston, 1990.

\bibitem[DV02]{DVInv}
L.~Di~Vizio.
\newblock Arithmetic theory of {$q$}-difference equations. {T}he {$q$}-analogue
  of {G}rothendieck-{K}atz's conjecture on {$p$}-curvatures.
\newblock {\em Inventiones Mathematicae}, 150(3):517--578, 2002.
\newblock arXiv:math.NT/0104178.

\bibitem[DVH10a]{diviziohardouin}
L.~Di~Vizio and C.~Hardouin.
\newblock {Algebraic and differential generic Galois groups for q-difference
  equations, followed by the appendix "The Galois D-groupoid of a q-difference
  system" by Anne Granier}.
\newblock Arxiv:1002.4839v4, 2010.

\bibitem[DVH10b]{diviziohardouinCRAS}
L.~Di~Vizio and C.~Hardouin.
\newblock Courbures, groupes de galois génériques et d-groupoïde de galois d'un
  système aux q-différences.
\newblock {\em Comptes Rendus Mathematique}, 348(17--18):951--954, 2010.

\bibitem[GGO11]{GilletGorchinskyOvchinnikov}
H.~Gillet, S.~Gorchinskiy, and A.~Ovchinnikov.
\newblock Parameterized {P}icard-{V}essiot extensions and {A}tiyah extensions,
  2011.

\bibitem[GO11]{OvchinnikovGorchinskiy}
S.~Gorchinskiy and A.~Ovchinnikov.
\newblock {C}ompletely {I}ntegrable {P}{D}{E}s {A}nd {D}ifferential {T}annakian
  {C}ategories.
\newblock {I}n preparation, 2011.

\bibitem[Gra11]{GranierFourier}
A.~Granier.
\newblock A {G}alois ${D}$-groupoid for $q$-difference equations.
\newblock {\em Annales de l'Institut Fourier}, 2011.
\newblock To appear.

\bibitem[Har08]{hardouincompositio}
C.~Hardouin.
\newblock Hypertranscendance des syst\`emes aux diff\'erences diagonaux.
\newblock {\em Compositio Mathematica}, 144(3):565--581, 2008.

\bibitem[HS08]{HardouinSinger}
C.~Hardouin and M.F. Singer.
\newblock Differential {G}alois theory of linear difference equations.
\newblock {\em Mathematische Annalen}, 342(2):333--377, 2008.

\bibitem[Kol73]{kolchindiffalg}
E.R. Kolchin.
\newblock {\em Differential algebra and algebraic groups}, volume~54 of {\em
  Pure and applied mathematics}.
\newblock Academic Press, New York and London, 1973.

\bibitem[Lan70]{Landodiffker}
B.A. Lando.
\newblock Jacobi's bound for the order of systems of first order differential
  equations.
\newblock {\em Trans. Amer. Math. Soc.}, 152:119--135, 1970.

\bibitem[Lev08]{Levin:difference}
A.~Levin.
\newblock {\em Difference algebra}, volume~8 of {\em Algebra and Applications}.
\newblock Springer, New York, 2008.

\bibitem[Mal05]{MalgrangeInvolutifPS}
B.~Malgrange.
\newblock {\em Syst\`emes diff\'erentiels involutifs}, volume~19 of {\em
  Panoramas et Synth\`eses [Panoramas and Syntheses]}.
\newblock Soci\'et\'e Math\'ematique de France, Paris, 2005.

\bibitem[Mau10]{MauGal}
A.~Maurischat.
\newblock Galois theory for iterative connections and nonreduced {G}alois
  groups.
\newblock {\em Transactions of the American Mathematical Society},
  362(10):5411--5453, 2010.

\bibitem[McG00]{mcgrail}
T.~McGrail.
\newblock The model theory of differential fields with finitely many commuting
  derivations.
\newblock {\em Journal of Symbolic Logic}, 65(2):885--913, 2000.

\bibitem[Ovc09]{ovchinnikovdifftannakian}
A.~Ovchinnikov.
\newblock Differential tannakian categories.
\newblock {\em Journal of Algebra}, 321(10):3043--3062, 2009.

\bibitem[Pie03]{piercenoncommder}
D.~Pierce.
\newblock Differential forms in the model theory of differential fields.
\newblock {\em Journal of Symbolic Logic}, 68(3):923--945, 2003.

\bibitem[PN11]{peonnietodiffcomp}
A.~Pe{\'o}n~Nieto.
\newblock On {$\sigma\delta$}-{P}icard-{V}essiot extensions.
\newblock {\em Communications in Algebra}, 39(4):1242--1249, 2011.

\bibitem[Pra86]{Praag2}
C.~Praagman.
\newblock Fundamental solutions for meromorphic linear difference equations in
  the complex plane, and related problems.
\newblock {\em Journal f\"ur die Reine und Angewandte Mathematik},
  369:101--109, 1986.

\bibitem[Pul08]{PulitaComp}
A.~Pulita.
\newblock {$p$}-adic confluence of {$q$}-difference equations.
\newblock {\em Compositio Mathematica}, 144(4):867--919, 2008.

\bibitem[Rob00]{robert}
A.M. Robert.
\newblock {\em A course in {$p$}-adic analysis}, volume 198 of {\em Graduate
  Texts in Mathematics}.
\newblock Springer-Verlag, 2000.

\bibitem[Sau04]{SauloyENS}
J.~Sauloy.
\newblock Galois theory of {F}uchsian {$q$}-difference equations.
\newblock {\em Annales Scientifiques de l'\'Ecole Normale Sup\'erieure.
  Quatri\`eme S\'erie}, 36(6):925--968, 2004.

\bibitem[Sin07]{singernoncommder}
M.F. Singer.
\newblock Model theory of partial differential fields: from commuting to
  noncommuting derivations.
\newblock {\em Proceedings of the American Mathematical Society},
  135(6):1929--1934 (electronic), 2007.

\bibitem[vdPS97]{vdPutSingerDifference}
M.~van~der Put and M.F. Singer.
\newblock {\em Galois theory of difference equations}.
\newblock Springer-Verlag, Berlin, 1997.

\bibitem[vdPS03]{vdPutSingerDifferential}
M.~van~der Put and M.F. Singer.
\newblock {\em Galois theory of linear differential equations}.
\newblock Springer-Verlag, Berlin, 2003.

\bibitem[Wib10]{Wibmchev}
M.~Wibmer.
\newblock A {C}hevalley theorem for difference equations.
\newblock arXiv:1010.5066, 2010.

\bibitem[Wib11]{wibmer2011existence}
M.~Wibmer.
\newblock Existence of $\partial$-parameterized {P}icard-{V}essiot extensions
  over fields with algebraically closed constants.
\newblock ArXiv:1104.3514, 2011.

\bibitem[Yaf01]{yaffenoncommder}
Y.~Yaffe.
\newblock Model completion of {L}ie differential fields.
\newblock {\em Annals of Pure and Applied Logic}, 107(1-3):49--86, 2001.

\end{thebibliography}

\end{document}